\documentclass[12pt]{article}
\usepackage{amssymb,amsmath,amsthm,bm}
\usepackage[colorlinks=true, urlcolor=blue,linkcolor=blue, citecolor=blue]{hyperref}
\usepackage{mathrsfs,verbatim}
\usepackage{caption,cleveref }
\usepackage{graphicx, float}
\usepackage{constants, color, mathtools, tikz, xcolor}
\usepackage{enumerate}
\usetikzlibrary{shapes.geometric}
\usepackage{pdfpages}
\usepackage{caption}
\usepackage[linesnumbered,ruled,vlined]{algorithm2e}
\usepackage{appendix}
\usepackage{array}
\usepackage{multirow}
\usepackage{float}
\usepackage{tabularx}

\usetikzlibrary{positioning}
\usetikzlibrary{arrows}

\usepackage{subcaption}

\usetikzlibrary{calc}

\usetikzlibrary{decorations.pathreplacing}

\tikzstyle{level 1}=[level distance=2.75cm, sibling distance=5.65cm]
\tikzstyle{level 2}=[level distance=3cm, sibling distance=2.75cm]
\tikzstyle{level 3}=[level distance=3.9cm, sibling distance=1.5cm]

\tikzstyle{bag} = [text width=10em, text centered] 
\tikzstyle{end} = [circle, minimum width=3pt,fill, inner sep=0pt]

\newcolumntype{L}{>{$}l<{$}}
\newcolumntype{C}{>{$}c<{$}}
\SetKwInput{KwInput}{Input}                
\SetKwInput{KwOutput}{Output}              

\begin{document}
	
	\newtheorem{problem}{Problem}[section]
	\newtheorem{exercise}{Exercise}[section]
	\newtheorem{theorem}{Theorem}[section]
	\newtheorem{corollary}{Corollary}[section]
	\newtheorem{conjecture}{Conjecture}[section]
	\newtheorem{proposition}{Proposition}[section]
	\newtheorem{lemma}{Lemma}[section]
	\newtheorem{definition}{Definition}[section]
	\newtheorem{example}{Example}[section]
	\newtheorem{remark}{Remark}[section]
	\newtheorem{solution}{Solution}[section]
	\newtheorem{case}{Case}[section]
	\newtheorem{condition}{Condition}[section]
	\newtheorem{assumption}{Assumption}
	\newtheorem{note}{Note}[section]
	\newtheorem{notes}{Notes}[section]
	\frenchspacing

	\title{Defective Ramsey Numbers and Defective Cocolorings in Some Subclasses of Perfect Graphs }
	\author{Yunus Emre Demirci\footnote{Department of Mathematics, Bo\u{g}azi\c{c}i University, 34342, Bebek, Istanbul, Turkey} \hspace{0.2in} T{\i}naz Ekim \footnote{Department of Industrial Engineering, Bo\u{g}azi\c{c}i University, 34342, Bebek, Istanbul, Turkey} \hspace{0.2in}
		Mehmet Akif Y{\i}ld{\i}z\footnote{Department of Mathematics, Bo\u{g}azi\c{c}i University, 34342, Bebek, Istanbul, Turkey} } \vspace{0.25in}
	
	\maketitle
	\begin{abstract}
In this paper, we investigate a variant of Ramsey numbers called defective Ramsey numbers where cliques and independent sets are generalized to $k$-dense and $k$-sparse sets, both commonly called $k$-defective sets. We focus on the computation of defective Ramsey numbers restricted to some subclasses of perfect graphs. Since direct proof techniques are often insufficient for obtaining new values of defective Ramsey numbers, we provide a generic algorithm to compute defective Ramsey numbers in a given target graph class. We combine direct proof techniques with our efficient graph generation algorithm to compute several new defective Ramsey numbers in perfect graphs, bipartite graphs and chordal graphs. We also initiate the study of a related parameter, denoted by $c^{\mathcal G}_k(m)$, which is the maximum order $n$ such that the vertex set of any graph of order at $n$ in a class $\mathcal{G}$ can be partitioned into at most $m$ subsets each of which is $k$-defective. We obtain several values for $c^{\mathcal G}_k(m)$ in perfect graphs and cographs.
	\end{abstract}
	
	\section{Introduction}
Ramsey Theory deals with the existence of some unavoidable structures as the number of vertices in a graph grows. In (classical) Ramsey numbers, we are interested in the minimum order of a graph which guarantees the existence of a clique or an independent set of given sizes.
It is well-known that the computation of Ramsey numbers is an extremely difficult task. One way to cope with this hardness is to consider Ramsey numbers in restricted graph classes. Ramsey numbers have been computed in planar graphs \cite{5}, in graphs with bounded degree \cite{6,7,8}, and in claw-free graphs \cite{9}. More recently, this approach has been also applied to perfect graphs and their subclasses in \cite{ramseygraphclassesHeggernes}. 

Defective Ramsey numbers are a variation of (classical) Ramsey numbers where sparse and dense sets are used instead of independent sets and cliques. A \textit{$k$-sparse $j$-set} is a set $S$ of $j$ vertices of a graph $G$ such that each vertex in $S$ has at most $k$ neighbors in $S$. A \textit{$k$-dense $i$-set} is a set $D$ of $i$ vertices of a graph $G$ that is $k$-sparse in the complement of $G$; or alternatively, each vertex in $D$ misses at most $k$ vertices of $D$ in its neighborhood. A set is said to be \textit{$k$-defective} if it is either $k$-sparse or $k$-dense. The \textit{defective Ramsey number} $R_k^\mathcal{G}(i,j)$ is the smallest $ n $ such that all graphs on $n$ vertices in the class $\mathcal{G} $ have either a $k$-dense $i$-set or a $k$-sparse $j$-set. A graph is said to be \textit{extremal} for  $R_k^\mathcal{G}(i,j)=n$ if it is a graph of order $n-1$ belonging to the class $\mathcal{G}$ and having neither $k$-dense $i$-set nor $k$-sparse $j$-set. If the defectiveness level $k=0$ (in which case the subscript 0 can be omitted), then it boils down to the classical Ramsey numbers. To establish that $R_k^\mathcal{G}(i,j)=n$ for some $n$, one should show two things: i) all graphs of order at least $n$ in $\mathcal{G}$ have either a $k$-dense $i$-set or a $k$-sparse $j$-set; ii) there is at least one extremal graph for $R_k^\mathcal{G}(i,j)=n$. We note that besides being interesting for its own, finding all extremal graphs for a Ramsey number can be also helpful in establishing the next (for $i+1$ or $j+1$ where all the other parameters are the same) Ramsey numbers.

Defective Ramsey numbers have been first introduced in \cite{mynhardt} under the name of 1-dependent Ramsey numbers. Further defective Ramsey numbers have been computed in \cite{defectiveRamseyJohnChappell} and \cite{ekim1} using direct proof techniques. As expected, the approach of establishing new defective Ramsey numbers using direct proof techniques has quickly attained its limits. As a remedy, in \cite{defectiveparameterTinazAhu} and  \cite{defectiveRamseyJohnChappell}, computer-assisted efficient graph generation methods have been combined with direct proof techniques to obtain new defective Ramsey numbers (and related parameters). Besides efficient graph generation algorithms, a recent trend is to consider defective Ramsey numbers in graph classes. This has been initiated in \cite{1defectiveperfectTinazOylumJohn} where some 1-defective Ramsey numbers in perfect graphs have been computed. Further graph classes have been examined in \cite{defectiveramseynumbers} from the perspective of defective Ramsey numbers; namely, formulas for all or most cases have been established for defective Ramsey numbers in forests, cacti, bipartite graphs, split graphs and cographs, and the missing cases have been formulated as conjectures.

There is a close relationship between Ramsey numbers and the cocoloring problem where the vertices of a graph are partitioned into independent sets and cliques. A similar relationship exists between defective Ramsey numbers and the defective cocoloring problem where the vertices of a given graph are partitioned into $k$-defective sets. Given a graph $G$, a \textit{$k$-defective $m$-cocoloring} of $G$ consists in a partition of the vertex set of $G$ into $m$ subsets such that each one is a $k$-sparse or a $k$-dense set. The parameter $c_k(m)$ is defined as the maximum order $n$ such that every $n$-graph has a $k$-defective $m$-cocoloring. The classical version of this parameter (when $k=0$) has been studied by Straights in \cite{StraightsFormula} where $c_0(3)=8$ is shown. Further values such as $c_1(2)=7$, $c_1(3)\geq11$ and $c_2(2)\geq9$ are obtained by Ekim and Gimbel  \cite{ekim1}. This parameter has been formally defined by Akdemir and Ekim in \cite{defectiveparameterTinazAhu} where the values $c_0(4)=12$, $c_1(3)=12$ and $c_2(2)=10$ have been established via computer-assisted proofs using an efficient graph generation algorithm. 

\textbf{Our contribution:}
In this paper, we first provide a generic algorithm to compute defective Ramsey numbers in various graph classes. In Section \ref{sec:algo}, we explain our algorithm which is based on the idea of constructing all graphs (starting from order one) that do not contain $k$-dense $i$-sets nor $k$-sparse $j$-sets, and that belong to the desired graph class. It takes as input the defectiveness level $k$ as well as the orders $i$ and $j$ for the $k$-dense and $k$-sparse sets, respectively. When used in a recursive way along with a recognition algorithm for the target graph class, it results in the value of the defective Ramsey number in that class and all extremal graphs for it. All of our codes and the extremal graphs we obtain are available at \href{https://github.com/yunusdemirci/DefectiveRamsey}{https://github.com/yunusdemirci/DefectiveRamsey} \cite{github}.

By combining our efficient graph generation algorithm with classical proof techniques, we extend the results in \cite{defectiveramseynumbers} and \cite{1defectiveperfectTinazOylumJohn} by establishing new defective Ramsey numbers in perfect graphs and bipartite graphs, and investigate chordal graphs which have not been addressed from this perspective to the best of our knowledge. In what follows, let $\mathcal{PG}$, $\mathcal{BIP}$ and $\mathcal{CH}$ denote the classes of perfect, bipartite and chordal graphs, respectively.

In Section \ref{sec:perfect}, we follow up the results in \cite{1defectiveperfectTinazOylumJohn} by establishing new defective Ramsey numbers in perfect graphs.  First, we give all extremal graphs for $R_1^{\mathcal{PG}}(4,8)$ (whose value has been already shown in \cite{1defectiveperfectTinazOylumJohn} without providing all extremal graphs), and then establish that $R_1^{\mathcal{PG}}(4,9)=19$ and $R_1^{\mathcal{PG}}(4,10)=22$. We also compute several $k$-defective Ramsey numbers in perfect graphs for $k\in\{2,3,4\}$ as reported in Tables \ref{table:2-defectiveperfect}, \ref{table:3-defectiveperfect} and \ref{table:4-defectiveperfect} respectively. To conclude, we conjecture that there exists $l\geq 3$ such that $R_1^{\mathcal{PG}}(4,j+l)\leq R_1^{\mathcal{PG}}(4,j)+3l$ for all $j\geq 1$.

Section \ref{sec:bip} is devoted to bipartite graphs. In \cite{defectiveramseynumbers}, all but exactly five 1-defective Ramsey numbers in bipartite graphs have been established and the conjecture $R_1^{\mathcal{BIP}}(4,j)=2j-1$ has been suggested for the five remaining values, namely for $j\in \{ 10,11,12,18,19\}$. In this paper, we show that this conjecture does not hold for $j=10$ and 11 by establishing that $R_1^{\mathcal{BIP}}(4,10)=18$ and $R_1^{\mathcal{BIP}}(4,11)=20$. Besides, we consider $k$-defective Ramsey numbers in bipartite graphs for $k\in\{2,3,4\}$ (for the first time) and provide several values in Tables \ref{table:2-defectivebipartite}, \ref{table:3-defectivebipartite} and \ref{table:4-defectivebipartite}.

In Section \ref{sec:chordal}, we consider $k$-defective Ramsey numbers in chordal graphs and report our findings in Tables \ref{table:1-defectivechordal}, \ref{table:2-defectivechordal}, \ref{table:3-defectivechordal} and  \ref{table:4-defectivechordal} for $k=1,2,3$ and 4 respectively. We also show that we have $R_k^{\mathcal{CH}}(i,k+2)=i$ for all $k\geq 1$ and $i\geq k+2$; $R_k^{\mathcal{CH}}(k+2,j)=j$ for all $k\geq 1$ and $j\geq k+2$; and $R_1^{\mathcal{CH}}(4,j)=2j-2$ for all $j\geq 3$ building up on a result on cactus graphs obtained in \cite{defectiveramseynumbers}. 

It is important to note that the maximum defective Ramsey numbers we compute are of order 20 for both chordal and bipartite graphs, and of order 22 for perfect graphs. Indeed, it would not be possible to compute these values by considering all (unlabeled) graphs in a target graph class and checking for $k$-dense $i$-sets and $k$-sparse $j$-sets. The most efficient enumeration algorithms, by B. McKay \cite{McKay}, list all (unlabeled) connected chordal graphs up to 13 vertices, and all (unlabeled) perfect graphs up to 11 vertices. The efficiency of our algorithm is ensured by two main reasons. First, we do not generate all graphs in the target class and then check if each one admits the desired properties. Instead, we eliminate graphs which are not candidate for being extremal graphs at early stages of the generation. On top of that, we perform the recognition algorithms for graph classes efficiently by taking advantage of the fact that the new graphs are obtained by adding only one vertex to the previously obtained (sub-extremal) graphs which are already known to belong to the target graph class.

Besides computing new defective Ramsey numbers in perfect graphs, bipartite graphs and chordal graphs, we also initiate the study of the parameter $c_k(m)$ in various graph classes in Section \ref{sec:coco}. Formally, we define $c^{\mathcal G}_k(m)$ as the maximum order $n$ such that every $n$-graph in the graph class $\mathcal G$ has a $k$-defective $m$-cocoloring. We first establish a lower bound and an upper bound for the parameter $c^{\mathcal G}_k(m)$. Then, in Section \ref{sec:cocoperfect}, we focus on perfect graphs and show that $c_0^{\mathcal{PG}}(m)=m(m+3)/2$ for any $m$. Making use of efficient graph generation algorithms, we also establish two values, namely $c_1^{\mathcal{PG}}(2)=7$ (with exactly 24 extremal graphs) and $c_2^{\mathcal{PG}}(2)=11$, as well as the bounds $13 \leq c_1^{\mathcal{PG}}(3) \leq 14$ and $12 \leq c_3^{\mathcal{PG}}(2) \leq 14$. Lastly, in Section \ref{sec:cococograph}, we investigate the class of cographs, denoted by $\mathcal{CO}$, using solely direct proof techniques. The main result of this section is that $c_k^{\mathcal{CO}}(2)=3k+5$ for all $k\geq 1$.
	
\section{Computer Based Search for Extremal Graphs}\label{sec:algo}

In this section, we describe a generic algorithm to examine defective Ramsey numbers in a graph class. In subsequent sections, we will use this algorithm to derive some new values of defective Ramsey numbers in perfect graphs, bipartite graphs and chordal graphs, as well as to find related extremal graphs.  The codes of Algorithm \ref{algo:CHECK} for perfect graphs, bipartite graphs and chordal graphs and the lists of all extremal graphs obtained in this paper are available in our github account \cite{github}.

Let $\mathcal{T}_n^\mathcal{G}(k,i,j)$ be the set of all graphs of order $n$ belonging to the graph class $\mathcal{G}$ and containing no $k$-dense $i$-set or $k$-sparse $j$-set. For any integer $n$, we will call a $k$-dense $i$-set or a $k$-sparse $j$-set \textit{a forbidden $k$-defective set for} $\mathcal{T}_n^\mathcal{G}(k,i,j)$. Note that the set of all extremal graphs for ${R}_k^\mathcal{G}(i,j)$ corresponds to the set $\mathcal T_{n}^\mathcal{G}(k,i,j)$ for $n={R}_k^\mathcal{G}(i,j)-1$. Accordingly, a graph in  $\mathcal{T}_{n}^\mathcal{G}(k,i,j)$ for $n < {R}_k^\mathcal{G}(i,j)-1$ will be  called a \textit{sub-extremal} graph for ${R}_k^\mathcal{G}(i,j)$. 

Our algorithm is based on the observation that every induced subgraph of an (sub-) extremal graph is a sub-extremal graph. Consequently, it starts with a set $W$ of graphs having no forbidden $k$-defective set for  $\mathcal{T}_n^\mathcal{G}(k,i,j)$ and produces the set $K$ of all graphs in $\mathcal{T}_{n+1}^\mathcal{G}(k,i,j)$ containing a graph in $W$ as induced subgraph. Clearly, if we start with $n=1$ and $W=\{K_1\}$, and run the algorithm recursively by taking the output set $K$ as input for the next run until the returned set $K$ is empty, then the last value of $n$ is equal to ${R}_k^\mathcal{G}(i,j)$ and the last non-empty output set $K$ is equal to the set of all extremal graphs for ${R}_k^\mathcal{G}(i,j)$. Alternatively, we can achieve the same goal by starting with the set of all sub-extremal graphs $\mathcal{T}_n^\mathcal{G}(k,i,j)$ for some $n<{R}_k^\mathcal{G}(i,j)$.\\

	\begin{algorithm}[H]
		
		\KwInput{$W\subseteq\mathcal{T}_n^\mathcal{G}(k,i,j)$, parameters $k, i, j$ such that $i,j\geq k+2$}
		\KwOutput{All graphs in $\mathcal{T}_{n+1}^\mathcal{G}(k,i,j)$ which contain at least one graph from $W$ as an induced subgraph}
		
		Let $K=\emptyset$. \\ 
		\For{$G\in W$}{
			\For{\label{line:S}$S\subseteq V(G)$}{
				Take the graph $G_S$ that is formed by adding a new vertex $v$ into $G$ and all edges between $v$ and all vertices in $S$.\\ \label{line:allS}
				Let $u=$ \textbf{TRUE}.\\ \label{line:candidate}
				\For{$I\subseteq V(G_S)$ such that $v\in I$ and $|I|\in\{i,j\}$}  { 
					\If{$|I|=i$ and $G[I]$ is $k$-dense}{$u=$ \textbf{FALSE} and \textbf{BREAK}}
					\If{$|I|=j$ and $G[I]$ is $k$-sparse}{$u=$ \textbf{FALSE} and \textbf{BREAK}}
				} \label{line:kdef}
				
				 \If{u}{\label{line:class} \If{$G_S\in\mathcal{G}$}{\label{line:G}Add $G_S$ into $K$.}} \label{line:K}
				
		}}
		Return a maximal non-isomorphic set of graphs in $K$.
		\caption{Sub-extremal} \label{algo:CHECK}
	\end{algorithm}
	
	\vspace{0.3cm}
	
Technically, Algorithm \ref{algo:CHECK} takes as input a set $W$ of graphs on $n$ vertices which has no $k$-dense $i$-set and no $k$-sparse $j$-set. For each graph $G$ in the input set $W$, in lines \ref{line:S} and \ref{line:allS}, it constructs a new graph $G_S$ by adding a new vertex $v$ to $G$ in every possible way. Then, in lines \ref{line:candidate} to \ref{line:kdef}, it examines all subsets of the new graph $G_S$ containing $v$, and eliminates those containing a forbidden $k$-defective set. In lines \ref{line:class} to \ref{line:K}, only those graphs belonging to the graph class $\mathcal G$ among the remaining ones are added to the output set $K$. Finally, an isomorphism check taken from \cite{defectiveRamseyJohnChappell} is applied to return a maximal set of non-isomorphic graphs in $K$. At the end, Algorithm \ref{algo:CHECK} gives all graphs on $n+1$ vertices in the studied graph class which has no forbidden $k$-defective set and having at least one graph from the input as an induced subgraph.
	
In the following remark, we summarize how Algorithm \ref{algo:CHECK} can be used to produce all extremal graphs and thus the related defective Ramsey number for all parameters $k,i$ and $j$. 

	\begin{remark}\label{rem:algo}
If we input $W=\mathcal{T}_n^\mathcal{G}(k,i,j)$, then Algorithm \ref{algo:CHECK} produces $\mathcal{T}_{n+1}^\mathcal{G}(k,i,j)$. For all integers $k,i$ and $j$, if we take $W=\mathcal{T}_1^\mathcal{G}(k,i,j)=\{K_1\}$ and apply Algorithm \ref{algo:CHECK} recursively until the set $K=\mathcal{T}_n^\mathcal{G}(k,i,j)$ is empty, then $R_k^{\mathcal G}(i,j)=n$ and $\mathcal{T}_{n-1}^\mathcal{G}(k,i,j)$ is the set of all extremal graphs for $R_k^{\mathcal G}(i,j)=n$.
	\end{remark}

All the algorithms in this paper are implemented in Python (except the improvement for the
lower bounds on $c_1^{\mathcal{PG}}(3)$ which is implemented in Julia) and executed on an Intel Core i7 machine with a 2.50-GHz clock speed and 8GB of RAM memory.

For the computation of 1-defective Ramsey numbers, we combined direct proof techniques with the use of Algorithm \ref{algo:CHECK} and set the time limit to two days for perfect graphs, one day for bipartite graphs and 12 hours for chordal graphs. For higher defectiveness values $k=2,3$ and 4, instead of setting a time limit, we rather limited the value of the defective Ramsey number to 15 in our tables. In most of the cases, these tables can be easily extended by allowing the algorithm to run longer; further details will be provided in related sections.

\section{Defective Ramsey Numbers in Perfect Graphs} \label{sec:perfect}

In this section, we use Remark \ref{rem:algo} to compute new defective Ramsey numbers in perfect graphs and determine their extremal graphs. In line \ref{line:G} of Algorithm \ref{algo:CHECK}, we use a straightforward algorithm for perfect graph recognition: according to the Strong Perfect Graph Theorem \cite{SPGT}, a graph $G$ is perfect if and only if neither $G$ nor its complement contains an \textit{odd hole}, that is, an induced cycle of odd length at least 5. To this end, we take all subsets of vertices of odd order at least five and eliminate the graph under consideration if such a set induces a cycle or the complement of a cycle. Note that we need to check perfectness only for graphs induced by subsets of vertices containing the newly added vertex since all the input graphs of Algorithm \ref{algo:CHECK} are known to be perfect by definition of $\mathcal{T}_n^\mathcal{PG}(k,i,j)$, thus they do not contain odd holes or their complements.

Our first result is related to $R_1^{\mathcal{PG}}(4,8)$; its value has been shown to be 15 in \cite{1defectiveperfectTinazOylumJohn}, however only one extremal graph (the Heawood graph depicted in Figure \ref{fig:IMAGEextremalfor4-1-8-1}) has been reported. Here, we show that there are exactly three extremal graphs.
	
	\begin{theorem}\label{thm:4-8-extremal-three}
	There are three extremal graphs for $R_1^{\mathcal{PG}}(4,8)=15$, shown in Figures \ref{fig:IMAGEextremalfor4-1-8-1}, \ref{fig:IMAGEextremalfor4-1-8-2}, and \ref{fig:IMAGEextremalfor4-1-8-3}.
	\end{theorem}
	\begin{figure}[h]
		\begin{minipage}{0.4\textwidth}

			\centering
			\includegraphics[scale=0.55]{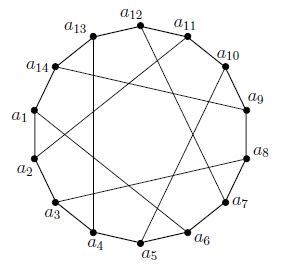}
			\caption{Extremal graph $G_1$ for $R_1^{\mathcal{PG}}(4,8)$.} \label{fig:IMAGEextremalfor4-1-8-1}

		\end{minipage}
		\hfill
		\begin{minipage}{0.5\textwidth}
			
			\centering
			\includegraphics[scale=0.65]{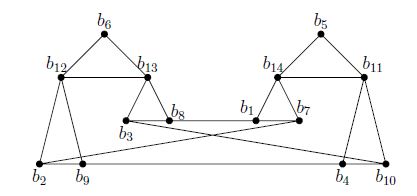}
			\caption{Extremal graph $G_2$ for $R_1^{\mathcal{PG}}(4,8)$.} \label{fig:IMAGEextremalfor4-1-8-2}
			
		\end{minipage}
		
		
		
	\end{figure}
	
	\begin{figure}[h]
		
		\centering
		\includegraphics[scale=0.6]{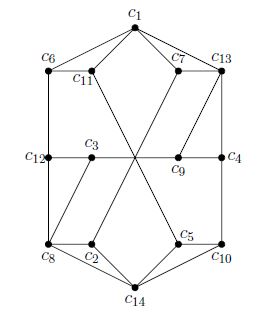}
		\caption{Extremal graph $G_3$ for $R_1^{\mathcal{PG}}(4,8)$.} \label{fig:IMAGEextremalfor4-1-8-3}
		
	\end{figure} 	
\begin{proof}
By Remark \ref{rem:algo}, we identified $\mathcal{T}_{14}^\mathcal{PG}(1,4,8)$, that is all extremal graphs for $R_1^{\mathcal{PG}}(4,8)$, by running Algorithm \ref{algo:CHECK} with parameters $k=1$, $i=4$, $j=8$ and starting with $W=\mathcal{T}_{1}^\mathcal{PG}(1,4,8)=\{K_1\}$. 
\end{proof}
	
	\noindent Our next result is obtained using a combination of direct proof techniques and computer assisted search. We show that $R_1^{\mathcal{PG}}(4,9)\leq 19$ by using theoretical approaches whereas the lower bound comes from the extremal graph found by Algorithm \ref{algo:CHECK}.  
	
	\begin{theorem}\label{thm:4-9}
		We have $R_1^{\mathcal{PG}}(4,9)=19$ with the unique extremal graph $G_4$ given in Figure \ref{fig:extremalfor4-1-9}.
	\end{theorem}
	
	\begin{proof}
		Firstly, take a perfect graph $G$ on $19$ vertices. We will show that it has either a $1$-dense $4$-set or a $1$-sparse $9$-set. If it has a $1$-dense $4$-set, we are done. So assume $G$ has no $1$-dense $4$-set. This implies in particular that $N(v)$ is 1-sparse for every vertex $v\in V(G)$. We will prove that $G$ has a $1$-sparse $9$-set by examining two cases:
		\begin{itemize}
			\item[(i)] If there exists a vertex $v$ of degree at most three, then $G-N[v]$ is a perfect graph on at least 15 vertices. By Theorem \ref{thm:4-8-extremal-three}, $G-N[v]$ has a $1$-sparse $8$-set, say $S$. Then, $S\cup\{v\}$ is a $1$-sparse $9$-set in $G$, we are done.
			\item[(ii)] If all vertices have degree at least four, take an arbitrary vertex $v$, and let $\{v_1,v_2,v_3,v_4\}$ be four neighbors of $v$. Since $G$ has no $1$-dense $4$-set, we have $N(v_i)\cap N(v_j)=\{v\}$ for all $i\neq j$ and  $N(v_i)-N[v]$ is a $1$-sparse set with at least two vertices (since $d(v_i)\geq 4$) for all $i=1,2,3,4$. Moreover, there are no edges between $N(v_i)-N[v]$ and $N(v_j)-N[v]$ since $G$ has no induced cycle of length five (by perfectness) or four (in case $v_iv_j\in E$, by the absence of 1-dense 4-sets). As a result, $\{v\}\cup\big(N(v_1)-N[v]\big)\cup\big(N(v_2)-N[v]\big)\cup\big(N(v_3)-N[v]\big)\cup\big(N(v_4)-N[v]\big) $ contains a $1$-sparse $9$-set, we are done.
			\end{itemize}
		
Secondly, we show that there is only one extremal graph for $R_1^{\mathcal{PG}}(4,9)$ with 18 vertices. Let $H$ be a perfect graph on 18 vertices which has no $1$-dense $4$-set and no $1$-sparse $9$-set. Note that if all vertices of $H$ have degree at least four, then $H$ has a 1-dense 4-set or a 1-sparse 9-set using the same arguments as in case ii), a contradiction. If $H$ has a vertex $v$ of degree at most 2, then we obtain a 1-dense 4-set or a 1-sparse 9-set using the same arguments as in case i). It follows that $H$ has a vertex $v$ of degree three such that $G-N[v]\in \mathcal{T}_{14}^\mathcal{PG}(1,4,8)$. By Theorem \ref{thm:4-8-extremal-three}, we have $ \mathcal{T}_{14}^\mathcal{PG}(1,4,8)=\{G_1, G_2, G_3\}$. So, every extremal graph for $R_1^{\mathcal{PG}}(4,9)$ with 18 vertices should contain one graph in $\{G_1, G_2, G_3\}$. Therefore, $H$ can be  obtained by running Algorithm \ref{algo:CHECK} recursively by setting the parameters $k=1$, $i=4$, $j=9$ and starting with $W=\{G_1, G_2, G_3\}=\mathcal{T}_{14}^\mathcal{PG}(1,4,8)\subseteq \mathcal{T}_{14}^\mathcal{PG}(1,4,9)$ until output graphs have 18 vertices. It turns out that Algorithm \ref{algo:CHECK} returns only one extremal graph $G_4$ with 18 vertices which is depicted in Figure \ref{fig:extremalfor4-1-9}. In $G_4$, we note vertex 18 has degree three and $G- N[18]$ is isomorphic to $G_3$.
\begin{figure}[htbp]
		\centering
		\includegraphics[width=6cm,angle=0,height=6cm]{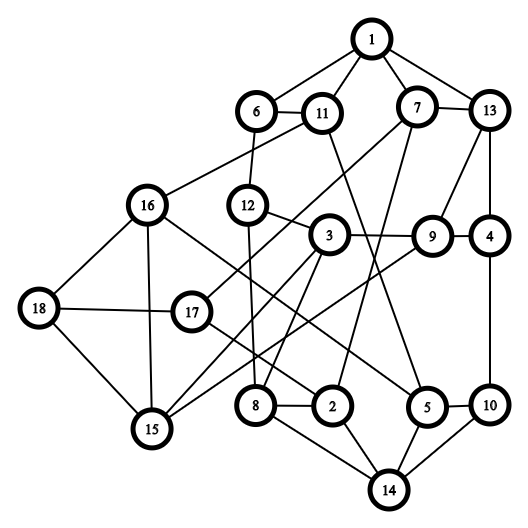}
		\caption{Unique extremal graph $G_4$ for $R_1^{\mathcal{PG}}(4,9)$.} \label{fig:extremalfor4-1-9}
	\end{figure}	
\end{proof}
		
We now compute $R_1^{\mathcal{PG}}(4,10)$. Again, we construct an extremal graph by using our computer based approach, and then we combine both theoretical results and Algorithm \ref{algo:CHECK} to find the desired value. This time, we do not provide the full list of extremal graphs, but exhibit only one.
	
	\begin{theorem}\label{thm:4-10}
		We have $R_1^{\mathcal{PG}}(4,10)=22$.
	\end{theorem}
	
	\begin{proof}
		Firstly, we run Algorithm \ref{algo:CHECK} by setting the parameters $k=1$, $i=4$, $j=10$ and starting with $W=\mathcal{T}_{18}^\mathcal{PG}(1,4,9)=\{G_4\}\subseteq \mathcal{T}_{18}^\mathcal{PG}(1,4,10)$ recursively until the output is empty. We observe that the output is empty for $n=22$, meaning that there is no graph in $\mathcal{T}_{22}^\mathcal{PG}(1,4,10)$ which contains $G_4$ as an induced subgraph. On the other hand, we obtain a unique graph in $\mathcal{T}_{21}^\mathcal{PG}(1,4,10)$ that contains $G_4$ as an induced subgraph; the graph $G_5$ given in Figure \ref{fig:extremalfor1-4-10} has 21 vertices, $G_5-\{19,20,21\}$ induce $G_4$, and it has no $1$-dense $4$-set or $1$-sparse $10$-set. It follows that $R_1^{\mathcal{PG}}(4,10)\geq22$. We note that since we do not start with the complete set $\mathcal{T}_{18}^\mathcal{PG}(1,4,10)$ but instead with only a subset of it, our method do not guarantee the full list of extremal graphs for $R_1^{\mathcal{PG}}(4,10)$. 

\begin{figure}[htbp]
	\centering
	\includegraphics[width=6cm,angle=0,height=6cm]{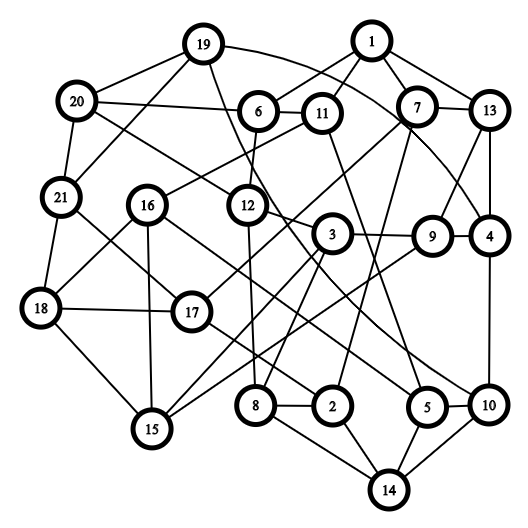}
	\caption{An extremal graph $G_5$ for $R_1^{\mathcal{PG}}(4,10)$.} \label{fig:extremalfor1-4-10}
\end{figure}

Now, let $G$ be a perfect graph on $22$ vertices. We will prove that $G$ has either a $1$-dense $4$-set or a $1$-sparse $10$-set. If $G$ has a 1-dense 4-set, we are done. So assume that $G$ has no $1$-dense $4$-set and we will prove that $G$ has a $1$-sparse $10$-set. As already stated, we know $G$ has no induced subgraph on $18$ vertices that is isomorphic to $G_4$, the unique graph in $\mathcal{T}_{18}^\mathcal{PG}(1,4,9)$. Thus, every induced subgraph of $G$ on $18$ vertices has a $1$-sparse $9$-set. Now, take a vertex $v\in V(G)$. If $d(v)\leq 3$, then $G-N[v]$ has at least $18$ vertices, thus contains a $1$-sparse $9$-set, say $J$. Then $J\cup\{v\}$ is a $1$-sparse $10$-set and we are done. Therefore, we assume that all degrees in $G$ are at least four. 
		\begin{itemize}
			\item[(i)] If there exists a vertex of degree at least five, say $v$, let $v_1$,$v_2$,$v_3$,$v_4$,$v_5$ be five neighbors of $v$. Since $G$ has no $1$-dense $4$-set, $\{v_1,v_2,v_3,v_4,v_5\}$ is $1$-sparse and any two of these five vertices have a unique common neighbor which is $v$. Moreover, for any $i\neq j$, we have that no neigbor of $v_i$ is adjacent to a neighbor of $v_j$ because $G$ has no induced cycle of lenght 5 (or 4). Since $d(v_i)\geq 4$ for all $i\in\{1,2,3,4,5\}$, the set $\bigcup_{i=1}^{5} \big(N(v_i)\backslash\{v\}\big)-\{v_1,v_2,v_3,v_4,v_5\}$ contains a $1$-sparse $10$-set, we are done. 
			\item[(ii)] If all vertices have degree four, take a vertex $v$ and its four neighbors $v_1$,$v_2$,$v_3$,$v_4$. Since ${S}=\{v,v_1,v_2,v_3,v_4\}$ has no $1$-dense $4$-set and $d(v_i)=4$ for all $i\in\{1,2,3,4\}$, $N(v_i)-{S}$ has at least two vertices for all $i$. If $|N(v_i)-{S}|\geq 3$ for some $i$, then $\bigcup_{i=1}^{4} N(v_i)-N(v)$ has a $1$-sparse $10$-set since $G$ has no $1$-dense $4$-set and no cycle of lenght $5$, we are done. So, assume $|N(v_i)-{S}|=2$ for all $i$, thus there are exactly two edges with both endpoints in $N(v)$, say without loss of generality $v_1v_2, v_3v_4\in E(G)$. Note that $v$ is selected arbitrarily, therefore we can suppose that $N(u)$ contains exactly two edges for all $u\in G$. Let $N(v_i)-\mathcal{S}=\{a_i,b_i\}$ for all $i$. Observe that $a_ib_i\in E(G)$ for all $i$ by our assumption. Moreover, since $G$ has no $1$-dense $4$-set and no cycle of size $5$, there are no edges between two sets $\{a_i,b_i\}$ and $\{a_j,b_j\}$ for all $i\neq j$. Consider the sets $U=\big(N(a_1)\cup N(b_1)\big)-\{v_1,a_1,b_1\}$ and $V=\big(N(a_2)\cup N(b_2)\big)-\{v_2,a_2,b_2\}$. Observe that $|U|=|V|=4$ and they are disjoint. Since $G$ has exactly 22 vertices, there exists a unique vertex in $V(G)-({S}\cup\{a_1,b_1,a_2,b_2,a_3,b_3,a_4,b_4\}\cup U\cup V)$, say $p$. Note that $p$ has no neighbors in ${S}\cup\{a_1,b_1,a_2,b_2\}$ because $G$ has no $1$-dense $4$-set and no cycle of size $5$. Since $d(p)=4$, by the pigeonhole principle, $p$ has at least two neighbors from one of the sets $\{a_3,b_3,a_4,b_4\}$, $U$ and $V$, which yields a $1$-dense $4$-set or a 5-cycle, a contradiction.    
		\end{itemize}
\end{proof}

We computed additional defective Ramsey numbers in perfect graphs by using Algorithm \ref{algo:CHECK}. Our results for defectiveness levels $k=1,2,3$ and 4 are reported in Tables \ref{table:1-defectiveperfect}, \ref{table:2-defectiveperfect}, \ref{table:3-defectiveperfect} and \ref{table:4-defectiveperfect}, respectively, with number of extremal graphs in parentheses (whenever we could obtain their full list).  

The values found in this study are highlighted in bold font. Values which are already known in Table \ref{table:1-defectiveperfect} are obtained in \cite{1defectiveperfectTinazOylumJohn}. For the remaining tables, the only values which are already known follow from Lemma 2.1 in \cite{defectiveramseynumbers} which states that $R_k^{\mathcal{PG}}(k+2,j) =j$ for all $j \geq k + 2$ (also $R_k^{\mathcal{PG}}(i,k+2)=i$ for all $i\geq k+2$, by self-complementarity of perfect graphs) where $k\in\{2,3,4\}$. While we execute our algorithm for two days when we deal with the values in the $1$-defective case, we could obtain all the values up to 15 in less than $4$ hours for $k$-defective cases for $k\geq 2$. In all computations, isomorphism checks are the most time consuming part as compared to checking perfectness and forbidden $k$-defective sets.\\

\begin{table}[H]\centering
			\scalebox{1}{
				\begin{tabular}{|C|C|C|C|C|C|C|C|C|}
					\hline
					\textbf{$R_1^{\mathcal{PG}}(i,j)$} & \textbf{3} & \textbf{4} & \textbf{5} & \textbf{6} & \textbf{7} & \textbf{8} & \textbf{9} & \textbf{10} \\
					\hline
					
					\textbf{3} & 3\textbf{(2)} & 4\textbf{(2)} & 5\textbf{(3)} & 6\textbf{(3)} & 7\textbf{(4)} & 8\textbf{(4)} & 9\textbf{(5)} & 10\textbf{(5)} \\
					\hline    
					
					\textbf{4} & 4\textbf{(2)} & 6(1) & 8(2) & 10(4) & 13(3) & 15\textbf{(3)} & \textbf{19(1)} & \textbf{22}\\
					\hline
					
					\textbf{5} & 5\textbf{(3)} & 8(2) & 13\textbf{(2)} &  &  & &&\\
					\hline
					
					
					
					
					
				\end{tabular}
			}
			
			\caption{1-Defective Ramsey Numbers in Perfect Graphs}
			\label{table:1-defectiveperfect}
		\end{table}
		\begin{table}[H]
			\centering
			\scalebox{1}{
				\begin{tabular}{|C|C|C|C|C|C|C|C|}
					\hline
					\textbf{$R_2^{\mathcal{PG}}(i,j)$} & \textbf{4} & \textbf{5} & \textbf{6} & \textbf{7} & \textbf{8} & \textbf{9} & \textbf{10}\\
					\hline
					
					\textbf{4} & 4\textbf{(4)} & 5\textbf{(4)} & 6\textbf{(4)} & 7\textbf{(4)} & 8\textbf{(4)} & 9\textbf{(4)} &10\textbf{(4)} \\
					\hline
					
					\textbf{5} & 5\textbf{(4)} & \textbf{7(2)} & \textbf{8(13)} & \textbf{10(16)} & \textbf{12(6)} & \textbf{15(2)}&\\
					\hline
					
					\textbf{6} & 6\textbf{(4)} & \textbf{8(13)} & \textbf{10(2)} & \textbf{13(7)} & & &\\
					\hline
					
					
					
					
					
				\end{tabular}
			}
			
			\caption{2-Defective Ramsey Numbers in Perfect Graphs}
			\label{table:2-defectiveperfect}
			
		\end{table}
		\begin{table}[H]
			
			\centering
			\scalebox{1}{
				\begin{tabular}{|C|C|C|C|C|C|C|C|}
					\hline
					\textbf{$R_3^{\mathcal{PG}}(i,j)$} & \textbf{5} & \textbf{6} & \textbf{7} & \textbf{8} & \textbf{9} & \textbf{10}& \textbf{11} \\
					\hline
					
					\textbf{5} & 5\textbf{(11)} & 6\textbf{(11)} & 7\textbf{(12)} & 8\textbf{(12)} & 9\textbf{(13)} & 10\textbf{(13)}& 11\textbf{(14)}\\
					\hline
					
					\textbf{6} & 6\textbf{(11)} & \textbf{8(4)} & \textbf{9(28)} & \textbf{10(159)} & \textbf{12(3)} & \textbf{13(67)} & \\
					\hline
					
					\textbf{7} & 7\textbf{(12)} & \textbf{9(28)} & \textbf{11(4)} & & & &\\
					\hline
					
					
					
					
					
				\end{tabular}
			}
			
			\caption{3-Defective Ramsey Numbers in Perfect Graphs}
			\label{table:3-defectiveperfect}
		\end{table}
		\begin{table}[H]
			\centering
			\scalebox{1}{
				\begin{tabular}{|C|C|C|C|C|C|C|C|}
					\hline
					\textbf{$R_4^{\mathcal{PG}}(i,j)$} & \textbf{6} & \textbf{7} & \textbf{8} & \textbf{9} & \textbf{10}& \textbf{11} & \textbf{12} \\
					\hline
					
					\textbf{6} & 6\textbf{(33)} & 7\textbf{(33)} & 8\textbf{(33)} & 9\textbf{(33)} & 10\textbf{(33)} & 11\textbf{(33)} & 12\textbf{(33)} \\
					\hline
					
					\textbf{7} & 7\textbf{(33)} & \textbf{9(11)} & \textbf{10(84)} & \textbf{11(549)} & \textbf{13(4)}& \textbf{14(28)} &\\
					\hline
					
					\textbf{8} & 8\textbf{(33)} & \textbf{10(84)} & \textbf{12(8)} & & & &\\
					\hline
					
					
					
					
					
				\end{tabular}
			}
			
			\caption{4-Defective Ramsey Numbers in Perfect Graphs}
			\label{table:4-defectiveperfect}
		\end{table}

We conclude this section with a conjecture on the growth of $R_1^{\mathcal{PG}}(4,j+3)$ for $j\geq 1$. A formula for the classical Ramsey numbers in perfect graphs has been provided in \cite{ramseygraphclassesHeggernes}. Namely, $R^{\mathcal{PG}}(i, j) = (i-1)(j-1) + 1$ for all $i, j \geq 1$. It follows that $R_0^{\mathcal{PG}}(4,j)=3j-2$. Since $R_1^{\mathcal{PG}}(4,j)\leq R_0^{\mathcal{PG}}(4,j)$, the value of $R_1^{\mathcal{PG}}(4,j)$ does not increase by more than three for consecutive vales of $j$ in the long run. However, the inequality $R_1^{\mathcal{PG}}(4,j+1)\leq R_1^{\mathcal{PG}}(4,j) +3$ does not hold in the light of our results $R_1^{\mathcal{PG}}(4,8)=15$ and $R_1^{\mathcal{PG}}(4,9)=19$. Likewise, $R_1^{\mathcal{PG}}(4,j+2)\leq R_1^{\mathcal{PG}}(4,j) +6$ does not hold neither since $R_1^{\mathcal{PG}}(4,8)=15$ and $R_1^{\mathcal{PG}}(4,10)=22$. Instead, we expect that the difference can be at most $3l$ in each $l$ steps for some fixed $l\geq 3$, which holds for currently known values.  Accordingly, we formulate the following conjecture.
	
	\begin{conjecture}
		There exists $l\geq 3$ such that $R_1^{\mathcal{PG}}(4,j+l)\leq R_1^{\mathcal{PG}}(4,j)+3l$ for all $j\geq 1$.
	\end{conjecture}

	\section{Defective Ramsey Numbers in Bipartite Graphs} \label{sec:bip}
	\noindent 
In \cite{defectiveramseynumbers}, it has been shown that $R_1^{\mathcal{BIP}}(4,j)=2j-1$ for all $j\geq 8$ except for $j\in \{ 10,11,12,18,19\}$. It has been also conjectured that for $j\in \{ 10,11,12,18,19\}$, we also have $R_1^{\mathcal{BIP}}(4,j)=2j-1$. In this section, we first establish two of these missing numbers, namely for $j=10$ and $j=11$  by combining classical proof techniques and Algorithm \ref{algo:CHECK}. It turns out that the conjecture reflecting the general trend does not hold for $j=10$ and $j=11$. Subsequently, we compute some new values of $k$-defective Ramsey numbers in bipartite graphs for $k\in \{2,3,4\}$, also by using Algorithm \ref{algo:CHECK}.

	\begin{theorem}
		$R_1^{\mathcal{BIP}}(4,10)=18$.
	\end{theorem}
	
	\begin{proof}
		From  \cite{defectiveramseynumbers}, we know $R_1^{\mathcal{BIP}}(4,9)=17$ which trivially implies $R_1^{\mathcal{BIP}}(4,10)\geq18$ since we can add an isolated vertex to the (unique) extremal graph of $R_1^{\mathcal{BIP}}(4,9)$ (which has 16 vertices) and obtain a bipartite graph which has no $1$-dense $4$-set and no $1$-sparse $10$-set. Now, we will show that any bipartite graph on $18$ vertices has either a $1$-dense $4$-set or a $1$-sparse $10$-set. Take a bipartite graph $G$ on $18$ vertices. If $G$ has a $1$-dense $4$-set, we are done, thus assume that $G$ has no $1$-dense $4$-set. We will prove that $G$ has a $1$-sparse $10$-set. Suppose there exists a vertex $v$ of degree at most $3$.  If $G-N[v]$ has a $1$-sparse $9$-set, say $J$, then $J\cup \{v\}$ is a $1$-sparse $10$-set. Otherwise, $G-N[v]\in\mathcal{T}_{m}^\mathcal{BIP}(1,4,9)$ for some $m\geq 14$. Noting that all graphs in $\mathcal{T}_{m}^\mathcal{BIP}(1,4,9)$ where $m\geq 14$ can be produced by the Algorithm \ref{algo:CHECK} by setting the parameters $k=1$, $i=4$, $j=10$ and starting with $W=\mathcal{T}_{14}^\mathcal{BIP}(1,4,9)\subseteq\mathcal{T}_{14}^\mathcal{BIP}(1,4,10)$. We run Algorithm \ref{algo:CHECK} with these settings and observe that there is no bipartite graph on $18$ vertices that has an induced subgraph from $\mathcal{T}_{m}^\mathcal{BIP}(1,4,9)$ for $m\geq14$. So, $G-N[v]$ has again a $1$-sparse $9$-set, say $J$, then $J\cup \{v\}$ is a $1$-sparse $10$-set and we are done.  

Now, assume that all vertices in $G$ have degree at least four. Take a vertex $u$ with its four neighbors $u_1$,$u_2$,$u_3$,$u_4$. Since $G$ has no $1$-dense $4$-set, $u$ is the unique common neighbor of any two of $x_1$, $x_2$, $x_3$, $x_4$. Moreover, since $G$ is bipartite, $N(x_1)\cup N(x_2)\cup N(x_3)\cup N(x_4)$ is an independent set (thus a 1-sparse set) of size at least $13$, which completes the proof.		
	\end{proof}

\begin{theorem}\label{thm:bipartite_1_4_11}
	$R_1^{\mathcal{BIP}}(4,11)=20$.
\end{theorem}

\begin{proof}
	From \cite{defectiveramseynumbers}, we know that there is a unique extremal graph, say $H$, on 16 vertices for $R_1^{\mathcal{BIP}}(4,9)=17$. Consider the disjoint union of $H$ and $K_{1,2}$. Clearly, this graph is a bipartite graph on 19 vertices and it has no $1$-dense $4$-set and no 1-sparse 11-set. Thus, $R_1^{\mathcal{BIP}}(4,11)\geq 20$. 
	
Now, let $G$ be a bipartite graph on 20 vertices and let us show that it has either a 1-dense 4-set or a 1-sparse 11-set. If it has a 1-dense 4-set, we are done. So, assume $G$ does not contain a 1-dense 4-set. Let $(A,B)$ be a bipartition of $G$. If $|A|\geq 11$ or $|B|\geq 11$, then we are done by taking $A$ or $B$, respectively, as a 1-sparse set. Hence, assume $|A|=|B|=10$. Besides, if $G$ has a vertex $w$ of degree at most one, wlog say $w\in A$, then $B\cup \{w\}$ is a 1-sparse 11-set. Thus, suppose all vertices in $G$ have degree at least two.
	
 Take a vertex $\omega\in G$ of maximum degree $d(w)=k$. Say wlog $\omega\in A$ and let $N(w)=\{x_1,x_2,...,x_k\}$ and $M_i=N(x_i)-\omega$. Assume wlog $|M_1|\leq|M_2|\leq\ldots\leq |M_k|$. Note that $M_i\cap M_j=\emptyset$ since $G$ has no $1$-dense $4$-set, and $|M_1|+|M_2|+\ldots+|M_k|\leq 9$. Hence, if $k\geq6$, we get $|M_1|=|M_2|=1$. Then $(A-w)\cup\{x_1,x_2\}$ is a $1$-sparse $11$-set and we are done. So, assume $k\leq 5$.
	
Now, if $G$ has no vertex of degree 3, then $|M_i|\in\{1,3,4\}$. Hence if $k\geq 4$, $|M_1|+|M_2|+|M_3|+|M_4|\leq 9$ implies $|M_1|=|M_2|=1$ and we are done as previously. Since $k\neq 3$, we have $k=2$ which means that  we obtain again $|M_1|=|M_2|=1$ and the result follows similarly. It follows that there exists a vertex $x$, say wlog in $A$, with $d(x)=3$. Let $N(x)=\{a,b,c\}$. 
	
We claim $G$ has two adjacent vertices both of which has degree at most three. If one of $\{a,b,c\}$ has the degree at most three, the claim holds. Assume $d(a),d(b),d(c)\geq 4$, since $N(a)-\{x\}$, $N(b)-\{x\}$ and $N(c)-\{x\}$ are pairwise disjoint subsets of $A-\{x\}$ and $|A-\{x\}|=9$, we get $|N(a)-\{x\}|=|N(b)-\{x\}|=|N(c)-\{x\}|=3$. Take a vertex $y\in B-\{a,b,c\}$. Since $G$ has no $1$-dense $4$-set, $y$ has at most one neighbor from each one of $N(a)-\{x\}$, $N(b)-\{x\}$, $N(c)-\{x\}$, thus $d(y)\leq 3$. Therefore, there are at most $3\cdot 7 + 4\cdot 3=33$ edges between $A$ and $B$. Then, by the pigeonhole principle, there exists $z\in A-\{x\}$ with $d(z)\leq 3$. Since $|N(z)\cap\{a,b,c\}|\leq 1$, $z$ should have a neighbor in $B-\{a,b,c\}$, and so the claim holds. As a result, $G$ has two adjacent vertices $u$ and $v$ such that $d(u)\leq 3$ and $d(v)\leq 3$.
	
Since all degrees in $G$ are at least two, there are three cases:
	\begin{enumerate}
		\item $d(u)=d(v)=2$. Let $c$ and $a$ be the unique neighbors of $u$ and $v$ other than $\{u,v\}$, respectively. Since $d(a)\geq 2$, $a$ has at least one neighbor other than $v$, take one of them and say $d$. Similarly, take a neighbor $b\neq a$ of $c$.
		\item One of $u$ and $v$ has degree three, the other has degree two, wlog say $d(u)=3$ and $d(v)=2$. Let $c$ and $d$ be the neighbors of $u$ other than $v$. Since $d(v),d(c)\geq 2$ and $G$ has no 1-dense 4-set, $v$ and $c$ have distinct neighbors, say $a$ and $b$, respectively. 
		\item $d(u)=d(v)=3$. Let $N(u)=\{v,c,d\}$ and $N(v)=\{u,a,b\}$.
	\end{enumerate}

\begin{figure}[h]
\centering
    \begin{tabularx}{1\textwidth}{*{3}{>{\centering\arraybackslash}X}}
\begin{tikzpicture}[
every edge/.style = {draw=black,very thick},
 vrtx/.style args = {#1/#2}{%
      circle, draw, thick, fill=white,
      minimum size=5mm, label=#1:#2}
                    ]
\node(U) [vrtx=center/$u$] at (-1, 2) {};
\node(V) [vrtx=center/$v$] at (1, 2) {};
\node(A) [vrtx=center/$a$] at (-1,1) {};
\node(C) [vrtx=center/$c$] at (1,1) {};
\node(B) [vrtx=center/$b$] at (-1,0) {};
\node(D) [vrtx=center/$d$] at (1,0) {};
\path   (U) edge (V)
(U) edge (C)
(A) edge (V)
(A) edge (D)	
(B) edge (C);
\end{tikzpicture}
    \caption{Case 1}  
\label{fig:case1}
    &   
\begin{tikzpicture}[
every edge/.style = {draw=black,very thick},
 vrtx/.style args = {#1/#2}{%
      circle, draw, thick, fill=white,
      minimum size=5mm, label=#1:#2}
                    ]
\node(U) [vrtx=center/$u$] at (-1, 2) {};
\node(V) [vrtx=center/$v$] at (1, 2) {};
\node(A) [vrtx=center/$a$] at (-1,1) {};
\node(C) [vrtx=center/$c$] at (1,1) {};
\node(B) [vrtx=center/$b$] at (-1,0) {};
\node(D) [vrtx=center/$d$] at (1,0) {};
\path   (U) edge (V)
(U) edge (C)
(A) edge (V)
(U) edge (D)	
(B) edge (C);
\end{tikzpicture}
\caption{Case 2}
\label{fig:case2}
    &   
\begin{tikzpicture}[
every edge/.style = {draw=black,very thick},
 vrtx/.style args = {#1/#2}{%
      circle, draw, thick, fill=white,
      minimum size=5mm, label=#1:#2}
                    ]
\node(U) [vrtx=center/$u$] at (-1, 2) {};
\node(V) [vrtx=center/$v$] at (1, 2) {};
\node(A) [vrtx=center/$a$] at (-1,1) {};
\node(C) [vrtx=center/$c$] at (1,1) {};
\node(B) [vrtx=center/$b$] at (-1,0) {};
\node(D) [vrtx=center/$d$] at (1,0) {};
\path   (U) edge (V)
(U) edge (C)
(A) edge (V)
(U) edge (D)	
(B) edge (V);
\end{tikzpicture}
\caption{Case 3}
\label{fig:case3}    
    \end{tabularx}
\end{figure} 
Since $G$ has no $1$-dense $4$-set, all edges present between $\{u,v,a,b,c,d\}$ in each case are given in Figures \ref{fig:case1}, \ref{fig:case2}, and \ref{fig:case3}, respectively. Now, if $G-\{u,v,a,b,c,d\}$ has a $1$-sparse $9$-set, say $S$, then $S\cup\{u,v\}$ is a 1-sparse $11$-set and we are done. So assume $G-\{u,v,a,b,c,d\}$ has no $1$-sparse $9$-set. Then we have $G-\{u,v,a,b,c,d\}\in \mathcal{T}_{14}^{\mathcal{BIP}}(1,4,9)$. Using Algorithm \ref{algo:CHECK}, we found all 73 graphs in $\mathcal{T}_{14}^{\mathcal{BIP}}(1,4,9)$, and examined all possible ways to combining these 73 graphs with one of the Cases 1, 2 and 3 to construct a bipartite graph on 20 vertices with no 1-dense 4-set. 
For all the resulting graphs, we could identify a $1$-sparse $11$-set, which proves that $R_1^{\mathcal{BIP}}(4,11)=20$.
\end{proof}

Using Algorithm \ref{algo:CHECK}, we also compute some $k$-defective Ramsey numbers in bipartite graphs for $k\in\{2,3,4\}$ where these values were not known previously. In \cite{defectiveramseynumbers}, it has been stated that $R_k^{\mathcal{BIP}}(i,j)$ was open for $k+3\leq i\leq 2k+2$ and $j\geq k+2$ when $k\geq 2$.  
Using Algorithm \ref{algo:CHECK}, we found some non-trivial values of $R_k^{\mathcal{BIP}}(k+3,j)$ for several $j$ values. Note that we did not use bold font to distinguish newly computed values since all values in Tables \ref{table:2-defectivebipartite}, \ref{table:3-defectivebipartite} and \ref{table:4-defectivebipartite} are new. All extremal graphs in these tables are available in our github account \cite{github}. Unlike for perfect graphs, we check bipartiteness just before checking forbidden defective sets since bipartite graphs can be recognized very efficiently as compared to perfect graphs. We note that in these calculations, we obtained defective Ramsey numbers up to 15 in less than five minutes, so the following tables can be extended by allowing the algorithm to run longer.

	\begin{theorem}
		The following hold where we denote the number of extremal graphs in parenthesis:

\begin{table}[H]\centering
			\scalebox{1}{
				\begin{tabular}{|C|C|C|C|C|C|}
					\hline
					\textbf{$R_2^{\mathcal{BIP}}(i,j)$} & \textbf{4} & \textbf{5} & \textbf{6} & \textbf{7} & \textbf{8}  \\
					\hline
					
					\textbf{5} & 5(1) & 6(4) & 8(1) & 10(2)& 11(56)\\
					\hline    
					
					\textbf{6} & 5(1) & 7(3) & 9(1) & 11(35) & 14(3)\\
					\hline

					\textbf{7} & 5(1) & 9(2) & 11(6) & 13(249) & \\
					\hline

					\textbf{8} & 5(1) & 9(2) & 11(6) & 13(249) & \\
					\hline

				\end{tabular}
			}
			\caption{2-Defective Ramsey Numbers in Bipartite Graphs}
			\label{table:2-defectivebipartite}
		\end{table}
		\begin{table}[H]

			\centering
			\scalebox{1}{
				\begin{tabular}{|C|C|C|C|C|C|C|}
					\hline
					\textbf{$R_3^{\mathcal{BIP}}(i,j)$} & \textbf{5} & \textbf{6} & \textbf{7} & \textbf{8} & \textbf{9}& \textbf{10}\\
					\hline
					
					\textbf{6} & 6(1) & 7(5) & 9(1) & 10(8) & 12(1) & 13(9) \\
					\hline
					
					\textbf{7} & 6(1) & 8(2) & 10(1) & 12(2)&14(26)&\\
					\hline
					
					\textbf{8} & 6(1) & 8(2) & 10(10) & 13(2)& 15(423)&\\
					\hline

					\textbf{9} & 6(1) & 8(2) & 13(5) & 15(40)& &\\
					\hline
				\end{tabular}

			}
			
			\caption{3-Defective Ramsey Numbers in Bipartite Graphs}
			\label{table:3-defectivebipartite}
		\end{table}
		\begin{table}[H]
			\centering
			\scalebox{1}{
				\begin{tabular}{|C|C|C|C|C|C|C|}
					\hline
					\textbf{$R_4^{\mathcal{BIP}}(i,j)$} & \textbf{6} & \textbf{7} & \textbf{8} & \textbf{9} & \textbf{10}  \\
					\hline
					
					\textbf{7} & 7(1) & 8(6) & 10(1) & 11(7) & 12(34) \\
					\hline
					
					\textbf{8} & 7(1) & 9(2) & 11(1) & 12(29) & 14(16) \\
					\hline
					
					\textbf{9} & 7(1) & 9(2) & 11(7) & 13(19) & 15(133) \\
					\hline

					\textbf{10} & 7(1) & 9(2) & 11(7) & 13(70) &  \\
					\hline
					
				\end{tabular}
			}
			
			\caption{4-Defective Ramsey Numbers in Bipartite Graphs}
			\label{table:4-defectivebipartite}
		\end{table}

\end{theorem}

As supported by the above tables, we obtain the following result which settles one of the aforementioned open cases for bipartite graphs pointed out in \cite{defectiveramseynumbers}.  
\begin{theorem}
For all $k\geq 1$ and $k+3\leq i\leq 2k+2$, we have $R_k^{\mathcal{BIP}}(i,k+2)=k+3$.
\end{theorem}
\begin{proof}
Consider a complete bipartite graph with one vertex in one part and $k+1$ vertices in the other part. It has no $k$-dense $i$-set (since there are less than $i$ vertices) and no $k$-sparse $(k+2)$-set (since there is a vertex of degree $k+1$). Now, take a bipartite graph of order at least $k+3$. If one part contains $k+2$ vertices, they form a $k$-sparse $(k+2)$-set. So assume each part contains at most $k+1$ vertices. Then, any subset of $k+2$ vertices containing at most $k$ vertices from each side forms a $k$-sparse set. 
\end{proof}

	\section{Defective Ramsey Numbers in Chordal Graphs}	\label{sec:chordal}

	\noindent In this section, we use Algorithm \ref{algo:CHECK} as described in Remark \ref{rem:algo} to compute new defective Ramsey numbers in chordal graphs and determine their extremal graphs. A graph is \textit{chordal} if it contains no induced  cycle of length four or more. Some characteristics of chordal graphs enables us to generate the set $\mathcal{T}_{n+1}^\mathcal{CH}(k,i,j)$ in an efficient way (from the input set $\mathcal{T}_n^\mathcal{CH}(k,i,j)$). A vertex is called \textit{simplicial} if its neighborhood induce a compete graph. A \textit{perfect elimination ordering} $\sigma$ in a graph is an ordering of the vertices of the graph such that, every vertex $v$ is simplicial in the graph induced by the vertices coming after $v$ in $\sigma$. It is known that every chordal graph has a simplicial vertex; moreover, a graph is chordal if and only if it has a perfect elimination ordering \cite{chordal_peo}.

\begin{lemma}\label{lem:chordal}
The set $\mathcal{T}_{n+1}^\mathcal{CH}(k,i,j)$ can be generated by adding only simplicial vertices to the graphs in the set $\mathcal{T}_n^\mathcal{CH}(k,i,j)$ in line \ref{line:allS} and without checking chordality in line \ref{line:G} of Algorithm \ref{algo:CHECK}. 
\end{lemma}
\begin{proof}
Let $G \in \mathcal{T}_n^\mathcal{CH}(k,i,j)$ and $G'$ be a graph in $\mathcal{T}_{n+1}^\mathcal{CH}(k,i,j)$ such that $G'$ contains $G$ as an induced subgraph. Let $x$ be the (unique) vertex in $V(G')\setminus V(G)$. Then, for any vertex $y\neq x$ in $G'$, the graph $G' -y \in \mathcal{T}_n^\mathcal{CH}(k,i,j)$ (since otherwise $G'$ would contain a $k$-dense $i$-set or $k$-sparse $j$-set, thus not belong to $\mathcal{T}_{n+1}^\mathcal{CH}(k,i,j)$). Since every chordal graph has a simplicial vertex, this also holds for a simplicial vertex $y\neq x$ in $G'$ (whose existence is guaranteed since $G$ is chordal), implying that all graphs in $\mathcal{T}_{n+1}^\mathcal{CH}(k,i,j)$ can be obtained by adding a simplicial vertex to a graph in $\mathcal{T}_{n}^\mathcal{CH}(k,i,j)$. Accordingly, in line \ref{line:allS} of Algorithm \ref{algo:CHECK}, it is sufficient to select only subsets of vertices forming a clique as the neighborhood of the newly added vertex. Moreover, all graphs constructed in this way are chordal since a perfect elimination ordering $\sigma$ of $G$, preceded by the newly added simplicial vertex $y$ is a perfect elimination order for $G'$. Therefore, the chordality check in line \ref{line:G} would always return a positive answer, thus it can be omitted. 
\end{proof}

Following Lemma \ref{lem:chordal}	, we can assume that the newly added vertex in line \ref{line:allS} is simplicial. Accordingly, we only consider subsets of vertices forming a clique when generating new graphs. Moreover, we do not check chordality in line \ref{line:G} as it is already ensured by Lemma \ref{lem:chordal}. We report the values of $k$-defective Ramsey numbers in chordal graphs for $k\in \{1,2,3,4\}$ in Tables \ref{table:1-defectivechordal}, \ref{table:2-defectivechordal},  \ref{table:3-defectivechordal} and  \ref{table:4-defectivechordal} respectively. We run Algorithm \ref{algo:CHECK} at most 12 hours to obtain the values in Table  \ref{table:1-defectivechordal} whereas each one of the   values (up to 15) in the remaining tables are obtained in at most 5 hours. 

	\begin{theorem}
		The following hold where we denote the number of extremal graphs in parenthesis:
\end{theorem}

	\begin{table}[H]\centering
			\scalebox{1}{
				\begin{tabular}{|C|C|C|C|C|C|C|C|C|C|C|C|}
					\hline
					\textbf{$R_1^{\mathcal{CH}}(i,j)$} & \textbf{3} & \textbf{4} & \textbf{5} & \textbf{6} & \textbf{7} & \textbf{8} & \textbf{9} & \textbf{10} & \textbf{11} \\
					\hline
					
					\textbf{3} & 3(2) & 4(2) & 5(3) & 6(3) & 7(4) & 8(4) & 9(5) & 10(5) &11(6)\\
					\hline    
					
					\textbf{4} & 4(2) & 6(1) & 8(1) & 10(1) & 12(1) & 14(1) & 16(1) & 18(1) & 20(1) \\
					\hline
					
					\textbf{5} & 5(2) & 7(4) & 10(4) & 12(44) & 15(18) &  & &&\\
					\hline
					
					\textbf{6} & 6(2) & 8(8) & 12(17) & 14(1397) & & &&&\\
					\hline
					
					\textbf{7} & 7(2) & 10(1) & 14 (68) & & & &&&\\
					\hline
					
					\textbf{8} & 8(2) & 11(2) & 16(293) & & & &&&\\
					\hline
					
					\textbf{9} & 9(2) & 12(4) & 18(1245) & & & & &&\\
					\hline
					
				\end{tabular}
			}
			
			\caption{1-Defective Ramsey Numbers in Chordal Graphs}
			\label{table:1-defectivechordal}
			
	\end{table}
	\begin{table}[H]

			\centering
			\scalebox{1}{
				\begin{tabular}{|C|C|C|C|C|C|C|C|C|C|C|C|}
					\hline
					\textbf{$R_2^{\mathcal{CH}}(i,j)$} & \textbf{4} & \textbf{5} & \textbf{6} & \textbf{7} & \textbf{8} & \textbf{9} & \textbf{10} \\
					\hline
					
					\textbf{4} & 4(4) & 5(4) & 6(4) & 7(4) & 8(4) & 9(4) &10(4) \\
					\hline
					
					\textbf{5} & 5(4) & 7(2) & 8(11) & 9(101) & 11(66) & 13(24)&\\
					\hline
					
					\textbf{6} & 6(4) & 8(8) & 10(2) & 12(45) & 14(92) & &\\
					\hline
					
					\textbf{7} & 7(4) & 9(22) & 11(50) & 14(316) & & &\\
					\hline
					
					\textbf{8} & 8(4) & 11(1) & 12(469) & & & &\\
					\hline
					
					\textbf{9} & 9(4) & 12(4) & 14(13) & & & & \\
					\hline
					
					\textbf{10} & 10(4) & 13(11) & 15(194) &  & & &\\
					\hline

				\end{tabular}
			}
			
			\caption{2-Defective Ramsey Numbers in Chordal Graphs}
			\label{table:2-defectivechordal}
		
	\end{table}
	\begin{table}[H]
			\centering
			\scalebox{1}{
				\begin{tabular}{|C|C|C|C|C|C|C|C|C|C|C|}
					\hline
					\textbf{$R_3^{\mathcal{CH}}(i,j)$} & \textbf{5} & \textbf{6} & \textbf{7} & \textbf{8} & \textbf{9} & \textbf{10} \\
					\hline

					\textbf{5} & 5(10) & 6(10) & 7(11) & 8(11) & 9(12) & 10(12)\\
					\hline
					
					\textbf{6} & 6(10) & 8(4) & 9(24) & 10(123) & 12(2) & 13(43) \\
					\hline
					
					\textbf{7} & 7(10) & 9(19) & 11(4) & 12(151) & 14(2) & \\
					\hline
					
					\textbf{8} & 8(10) & 10(62) & 12(124) & 14(7) & &\\
					\hline
					
					\textbf{9} & 9(10) & 12(2) & 13(1846) & & & \\
					\hline

				\end{tabular}
			}
			
			\caption{3-Defective Ramsey Numbers in Chordal Graphs}
			\label{table:3-defectivechordal}

	\end{table}
	\begin{table}[H]
			\centering
			\scalebox{1}{
				\begin{tabular}{|C|C|C|C|C|C|C|C|C|C|C|}
					\hline
					\textbf{$R_4^{\mathcal{CH}}(i,j)$} & \textbf{6} & \textbf{7} & \textbf{8} & \textbf{9} & \textbf{10} & \textbf{11} \\
					\hline

					\textbf{6} & 6(27) & 7(27) & 8(27) & 9(27) & 10(27) & 11(27)\\
					\hline
					
					\textbf{7} & 7(27) & 9(10) & 10(64) & 11(360) & 13(4) & 14(24)\\
					\hline
					
					\textbf{8} & 8(27) & 10(53) & 12(8) & 13(364) & &\\
					\hline
					
					\textbf{9} & 9(27) & 11(207) & 13(322) & & &\\
					\hline
					
					\textbf{10} & 10(27) & 13(4) &  &  & &\\
					\hline

				\end{tabular}
			}
			
			\caption{4-Defective Ramsey Numbers in Chordal Graphs}
			\label{table:4-defectivechordal}
		\end{table}

In what follows, we prove that the pattern that we observe for $R_k^{\mathcal{CH}}(k+2,j)$, $R_k^{\mathcal{CH}}(i,k+2)$ and  $R_1^{\mathcal{CH}}(4,j)$ in Tables \ref{table:1-defectivechordal}, \ref{table:2-defectivechordal}, \ref{table:3-defectivechordal}, \ref{table:4-defectivechordal} holds for all $i$ and $j$. Moreover, we describe the unique extremal graph for $R_1^{\mathcal{CH}}(4,j)$ for $4\leq j \leq 11$. 

\begin{lemma}\label{lemmaempty} \cite{defectiveramseynumbers}
Let $\mathcal{G}$ be a graph class containing all empty graphs. Then,
$$R_{k}^{\mathcal{G}}(k+2, j)=j \text { for all } j \geq k+2.$$ 
\end{lemma}

Since an empty graph is chordal, Lemma \ref{lemmaempty} implies the following:
\begin{remark}
 $R_k^{\mathcal{CH}}(k+2,j)=j$ for all $k\geq 1$ and $j\geq k+2$. 
\end{remark}

By taking the complementary class, the following is implied by Lemma \ref{lemmaempty}. 

\begin{corollary}\label{lemmacomplete} 
Let $\mathcal{G}$ be a graph class containing all complete graphs. Then,
$$R_{k}^{\mathcal{G}}(i, k+2)=i \text { for all } i \geq k+2.$$
\end{corollary}

Since a complete graph is chordal, Lemma \ref{lemmacomplete} implies the following:
\begin{remark}
 $R_k^{\mathcal{CH}}(i,k+2)=i$ for all $k\geq 1$ and $i\geq k+2$. 
\end{remark}

\begin{theorem}
$R_1^{\mathcal{CH}}(4,j)=2j-2$ for all $j\geq 3$. 
\end{theorem}

\begin{proof}
Let $G$ be a chordal graph on at least $2j-2$ vertices. If $G$ has a 1-dense 4-set, we are done. If not, $G$ does not contain a cycle of size 4 (as a partial subgraph) since it is 1-dense 4-set. Then, all 2-connected components of $G$ are either triangles or edges; thus, it is a \textit{cactus} graph as any two cycles of it have no edge in common. Let $\mathcal{CA}$ denote the class of cactus graphs. From  \cite{defectiveramseynumbers}, we know that $R_1^{\mathcal{CA}}(4,j)=2j-2$ for all $j\geq3$. It follows that $G$ contains a 1-sparse $j$-set, and we are done for all $j\geq4$. We conclude the proof with the extremal graph (see Figure \ref{fig:extremalforChordal}) that has been constructed for cactus graphs in \cite{defectiveramseynumbers}. This is a graph on $2j-3$ vertices with neither 1-dense 4-set nor 1-sparse $j$-set for all $j\geq3$. 
\end{proof}

\def\r{4pt}
\def\dy{1cm}
\tikzset{c/.style={draw,circle,fill=white,minimum size=\r,inner sep=0pt,
		anchor=center},
	d/.style={draw,circle,fill=black,minimum size=\r,inner sep=0pt, anchor=center}}

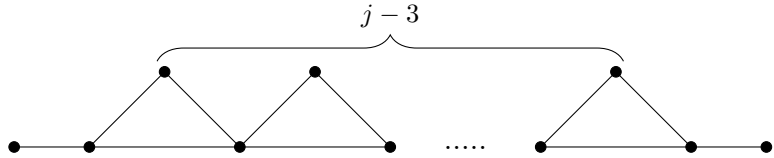
\begin{figure}[h]
	\begin{center}
		
		\begin{tikzpicture}

		\node[d] at (0,2) {};
		\node[d] at (1,1) {};
		\node[d] at (-1,1) {};
		\draw (0,2) to (1,1)
		(0,2) to (-1,1)
		(-1,1) to (1,1);
		
		\node[d] at (2,2) {};
		\node[d] at (3,1) {};
		\draw (1,1) to (2,2)
		(1,1) to (3,1)
		(3,1) to (2,2);
		
		\node[thick,minimum size=3cm] at (4,1) {.....};
		
		\node[d] at (6,2) {};
		\node[d] at (5,1) {};
		\node[d] at (7,1) {};
		\draw (5,1) to (6,2)
		(7,1) to (6,2)
		(7,1) to (5,1);
		
		\node[d] at (8,1) {};
		\draw (7,1) to (8,1);
		
		\node[d] at (-2,1) {};
		\draw (-1,1) to (-2,1);
		
		\draw [decorate,decoration={brace,amplitude=10pt},xshift=0pt,yshift=4pt]
		(-0.1,2) -- (6.1,2) node [black,midway,yshift=0.6cm] 
		{\footnotesize $j-3$};

		\end{tikzpicture}
		\vspace{-1.5cm}
	\end{center}
	\caption{An extremal graph for $R_1^{\mathcal{CH}}(4,j)=2j-2$ for $j\geq 3$ with $2j-3$ vertices.}\label{fig:extremalforChordal}
\end{figure}

According to Table \ref{table:1-defectivechordal} obtained by Algorithm \ref{algo:CHECK}, the extremal graph described in Figure \ref{fig:extremalforChordal} is the unique extremal graph for  $R_1^{\mathcal{CH}}(4,j)=2j-2$ for all $4\leq j\leq 11$. However, we do not know if this trend continues for all $j \geq 12$.

	\section{Defective Cocolorings} \label{sec:coco}
	
	\noindent  In this section, we study the parameter $c_k^{\mathcal{G}}(m)$ for perfect graphs and cographs. Recall that  $c_k^{\mathcal{G}}(m)$ is the maximum integer $n$ such that for every graph $G$ on $n$ vertices in the graph class $\mathcal{G}$, the vertices of $G$ can be partitioned into $m$ subsets where each subset is $k$-defective (either $k$-dense or $k$-sparse). In other words, every graph in a class $\mathcal G$ on at most $c_k^{\mathcal{G}}(m)$ vertices can be $k$-defectively cocolored with at most $m$ colors. In this section, an \textit{extremal graph for $c_k^{\mathcal{G}}(m)=n$} is a graph of order $n+1$ that belongs to the class $\mathcal G$ and whose vertices can not be partitioned into at most $m$ sets each of which is $k$-defective.

We start with some general observations that will be useful when studying defective cocolorings in perfect graphs and in cographs.  Since general graphs contain all graph classes, the following is an immediate consequence : 
\begin{remark}
For any graph class $\mathcal G$ and all integers $k\geq 0$, $m\geq 1$, we have $c_k(m)\leq c^{\mathcal{G}}_k(m)$.
\end{remark}

This remark together with previously known results guide us through our research. From \cite{defectiveparameterTinazAhu}, we know $c_k(1)=k+2$, $c_1(2)=7$, $c_2(2)=10$, and $c_1(3)=12$. In the same paper, it is also conjectured that $c_k(2)=3k+4$. 
We first generalize the formula $c_k(1)=k+2$ to graph classes.
	
\begin{lemma}\label{lem:c_k_1_G}
For all integers $k\geq 0$ and all graph classes $\mathcal G$ containing $K_{1,k+2}$, we have $c_k^{\mathcal{G}}(1)=k+2$.
\end{lemma} 
\begin{proof}
The result follows by observing that $K_{1,k+2}\in\mathcal{G}$ is not $k$-defective and each set of size $k+2$ is $k$-defective. 
\end{proof}

In \cite{defectiveparameterTinazAhu}, Straight's formula from \cite{StraightsFormula} has been generalized to the defective version as follows: $c_k(m-1)+t\leq c_k(m)\leq c_k(m-1)+(m+1)(k+1)$ where $t$ is any positive integer satisfying $c_k(m-1)+t\geq R_k(t,t)$. Next, we adapt the lower bound to graph classes.

\begin{lemma}\label{lem:Straightlowerbound}
If $c_k^{\mathcal{G}}(m-1)+t\geq R_k^{\mathcal{G}}(t,t)$ for some integers $k\geq 0$, $m\geq 2$, $t\geq 1$ and graph class $\mathcal{G}$, then we have $c_k^{\mathcal{G}}(m-1)+t\leq c_k^{\mathcal{G}}(m)$.
\end{lemma}	

\begin{proof}
Let us take a graph $G\in\mathcal{G}$ on $c_k^{\mathcal{G}}(m-1)+t$ vertices where $c_k^{\mathcal{G}}(m-1)+t\geq R_k^{\mathcal{G}}(t,t)$. By definition of $R_k^{\mathcal{G}}(t,t)$, $G$ has a $k$-defective $t$-set, say $T$. Now, $G-T$ has $c_k^{\mathcal{G}}(m-1)$ vertices, so it can be colored with $m-1$ colors where each color class is a $k$-defective set, which completes the proof.
\end{proof}

Secondly, we show that the upper bound of the above inequality can be adapted to graph classes with an additional property on the class. We use the same idea as in \cite{defectiveparameterTinazAhu} with a slight modification, which yields an improvement.

\begin{lemma}\label{lem:Straightupperbound}
	Let $\mathcal{G}$ be a graph class that is closed under taking the disjoint union with any clique. Then we have $c_k^{\mathcal{G}}(m)\leq c_k^{\mathcal{G}}(m-1)+m(k+1)+1$.
\end{lemma}	

\begin{proof}
	Take a graph class $\mathcal{G}$ with desired property. Let us construct a graph lying in $\mathcal{G}$ and having $c_k^{\mathcal{G}}(m-1)+m(k+1)+2$ vertices which cannot be $k$-defectively cocolored  using at most $m$ colors. By definition of $c_k^{\mathcal{G}}(m-1)$, there is a graph $H\in \mathcal G$ on $1+c_k^{\mathcal{G}}(m-1)$ vertices that cannot be partitioned into $m-1$ many $k$-defective sets. Consider the disjoint union of $K_{1+m(k+1)}$ and $H$, say $G$. Note that $G\in\mathcal{G}$ by definition of $G$, and it has $c_k^{\mathcal{G}}(m-1)+m(k+1)+2$ vertices. Assume $G$ can be partitioned into $m$ subsets such that each subset is $k$-defective. By the pigeonhole principle, there is a subset $S$ for which $|S\cap K_{1+m(k+1)}|\geq k+2$. Since $S$ is a $k$-defective set and each vertex in $S\cap K_{1+m(k+1)}$ has at least $k+1$ neighbors in $S$, it follows that $S$ is not $k$-sparse. So, $S$ is $k$-dense; therefore it contains no vertex from $H$ or else it would miss at least $k+2$ vertices of $S \cap K_{1+m(k+1)}$. As a result, $G-S$ contains $H$ and can to be partitioned into $m-1$ many $k$-defective sets, which is a contradiction since $H$ itself already requires at least $m$ colors.
\end{proof}

Now, we are ready to discuss the defective cocolorings in graph classes.

\subsection{Perfect Graphs}\label{sec:cocoperfect}
	
Firstly, we establish the formula for the classical cocoloring in perfect graphs where the defectiveness level is zero. 	
	\begin{theorem}\label{thm:c_0_m_PG}
		For all integers $m\geq 1$, we have $c_0^{\mathcal{PG}}(m)=\dfrac{m(m+3)}{2}$.
	\end{theorem}
	
	\begin{proof}
		Firstly, consider the disjoint union of $K_i$ where $i\in\{1,2,...,m+1\}$. It is a perfect graph on $1+2+...+(m+1)=\dfrac{(m+1)(m+2)}{2}=\dfrac{m(m+3)}{2}+1$ vertices. It can be easily seen that this graph cannot be partitioned into $m$ cliques or an independent sets. Secondly, let us show that every perfect graph on $\dfrac{m(m+3)}{2}$ vertices can be colored with $m$ colors such that each color class is either a clique or an independent set. We prove this by induction on $m$. The claim is trivial for $m=1$. Assume it holds for smaller values of $m$, and take a perfect graph $G$ on $\dfrac{m(m+3)}{2}$ vertices. If $\chi(G)\leq m$, then we are done. Assume $\chi(G)\geq m+1$, we get $\omega(G)\geq m+1$ since $G$ is perfect. Let us color a maximum clique with color $m$. Since the number of remaining vertices is at most $\dfrac{m(m+3)}{2}-(m+1)=\dfrac{(m-1)(m+2)}{2}$, they can be cocolored with $m-1$ colors by the induction hypothesis, hence we are done.
	\end{proof}
	
We can also note that we have $c_k^{\mathcal{PG}}(1)=k+2$ by Lemma \ref{lem:c_k_1_G} since all star graphs are perfect. We proceed with the computation of our parameter for non-trivial cases. We use a combination of theoretical analysis and computer assisted approach.
	
	\begin{theorem}\label{thm:c_1_2_PG}
		We have $c_1^{\mathcal{PG}}(2)=7$ with 24 extremal graphs.
	\end{theorem}
	
	\begin{proof}
		We note that $c_1^{\mathcal{PG}}(2)\geq c_1(2)=7$. Now, let us show that there are $24$ extremal graphs; that is perfect graphs on 8 vertices that can not be 1-defectively cocolored using at most 2 colors. Take a perfect graph on 8 vertices. If it has a 1-defective set of size 5, then two colors are enough since any set of three vertices is 1-defective. So any extremal graph on 8 vertices is free of 1-defective 5-sets, in other words, it belongs to the set  $\mathcal{T}_{8}^\mathcal{PG}(1,5,5)$. Accordingly, we generated the set $\mathcal{T}_{8}^\mathcal{PG}(1,5,5)$ using Algorithm \ref{algo:CHECK} starting with the single vertex graph as input. For each one of the resulting 824 graphs in $\mathcal{T}_{8}^\mathcal{PG}(1,5,5)$, we checked all possible partitions into two sets of four vertices each, and found that exactly 24 many of them cannot be partitioned into two 1-defective sets. The list of these 24 extremal graphs for $c_1^{\mathcal{PG}}(2)=7$ can be found in our github account \cite{github}.
	\end{proof}

Even though we could not find the number of extremal graphs, we get the exact value for $k=m=2$. 	
	
	\begin{theorem}\label{thm:c_2_2_PG}
		We have $c_2^{\mathcal{PG}}(2)=11$.
	\end{theorem}
	
	\begin{proof}
	By Lemmas \ref{lem:c_k_1_G} and \ref{lem:Straightupperbound}, we have $c_2^{\mathcal{PG}}(2)\leq c_2^{\mathcal{PG}}(1)+7=11$. Let us now show that all perfect graphs on 11 vertices can be partitioned into two $2$-defective sets. Take a perfect graph $G$ on $11$ vertices. Since $R_2^{\mathcal{PG}}(6,6)=10$ from \cite{defectiveramseynumbers}, $G$ has a $2$-defective $6$-set, say $S$.  If $G-S$ is $2$-defective, we are done. So, assume $G-S$ is not 2-defective. Let $\mathcal{A}$ be the set of all perfect graphs on 6 vertices that are $2$-defective, and $\mathcal{B}$ be the set of all graphs on 5 vertices that are not $2$-defective.  Then the vertex set of $G$ can be partitioned into two parts, one belonging to $\mathcal{A}$ and the other one belonging to $\mathcal{B}$. Using a computer, we enumerated all possible graphs obtained as a combination of two graphs $A$ and $B$ for some $A\in \mathcal{A}$ and $B\in \mathcal{B}$. We checked all such graphs and concluded that they can all be partitioned into two 2-defective sets, hence the desired result. 
\end{proof}
	
	
Next, we investigate $c_1^{\mathcal{PG}}(3)$ and $c_3^{\mathcal{PG}}(2)$. For these parameters, we provide bounds with two or three possible values, leaving the computation of the exact value open.

\begin{theorem}\label{thm:c_bounds_PG}
We have $13\leq c_1^{\mathcal{PG}}(3)\leq 14$ and $12\leq c_3^{\mathcal{PG}}(2)\leq 14$.
\end{theorem}

\begin{proof}
 Using Lemma \ref{lem:Straightupperbound} and Theorem \ref{thm:c_1_2_PG}, we get $c_1^{\mathcal{PG}}(3)\leq c_1^{\mathcal{PG}}(2)+7=14$. For the lower bound, we note that $c_1(3)=12$ from \cite{defectiveparameterTinazAhu}, which implies $12\leq c_1^{\mathcal{PG}}(3)$. Then, we show by computer enumeration that  $c_1^{\mathcal{PG}}(3)\neq12$. Indeed, $c_1^{\mathcal{PG}}(3)=12$ would imply that there is a perfect graph $G$ on 13 vertices which cannot be partitioned into three $1$-defective sets. Besides, since $R_1^{\mathcal{PG}}(5,5)=13$ from \cite{1defectiveperfectTinazOylumJohn}, $G$ has a $1$-defective $5$-set , say $J$. Then $G-J$ is a perfect graph on 8 vertices which can not be 1-defectivelty cocolored with 2 colors. It follows from Theorem \ref{thm:c_1_2_PG} that $G-J$ is one of the 24 extremal graphs for $c_1^{\mathcal{PG}}(2)=7$. Hence, $G$ is obtained as a combination of a 1-defective 5-set and one of these 24 extremal graphs. We checked all possible combinations (which took more than a month) and concluded that all such graphs can be partitioned into three $1$-defective sets, which yields $c_1^{\mathcal{PG}}(3)\neq12$.

As for $c_3^{\mathcal{PG}}(2)$, we have $R_3^{\mathcal{PG}}(7,7)=11$ from \cite{defectiveramseynumbers} and $c_3^{\mathcal{PG}}(1)=5$ from Lemma \ref{lem:c_k_1_G}. Thus, the inequality $c_k^{\mathcal{PG}}(m-1)+t\geq R_k^{\mathcal{PG}}(t,t)$ holds for $k=3$, $m=2$ and $t=7$, which gives $12=c_3^{\mathcal{PG}}(1)+7\leq c_3^{\mathcal{PG}}(2)$ by Lemma \ref{lem:Straightlowerbound}. On the other hand, by using Lemmas \ref{lem:Straightupperbound} and \ref{lem:c_k_1_G}, we get $c_k^{\mathcal{PG}}(2)\leq c_k^{\mathcal{PG}}(1)+1+2(k+1)=3k+5$ for all $k$, which implies $c_3^{\mathcal{PG}}(2)\leq 14$.
\end{proof}

%

\subsection{Cographs} \label{sec:cococograph}
\textit{Cographs}, denoted by $\mathcal{CO}$, is the class of graphs containing no induced path on four vertices. We start by noting that for the classical case ($k=0$), we have $c_0^{\mathcal{CO}}(m)=c_0^{\mathcal{PG}}(m)$ since the extremal graph constructed in the proof of Theorem \ref{thm:c_0_m_PG} is also a cograph and since $\mathcal{CO}\subset \mathcal{PG}$ we have $c_k^{\mathcal{PG}}(m)\leq c_k^{\mathcal{CO}}(m)$ for all $k$ and $m$.
   
\begin{remark}\label{rem:c_0_CO}
We have $c_0^{\mathcal{CO}}(m)=\dfrac{m(m+3)}{2}$.
\end{remark} 

\noindent Secondly, we have $c_k^{\mathcal{CO}}(1)=k+2$ from Lemma \ref{lem:c_k_1_G}. Then, the next candidates for examination are $c_1^{\mathcal{CO}}(2)$ and $c_2^{\mathcal{CO}}(2)$. Both values can be found as a natural consequence of results in perfect graphs.

\begin{theorem}\label{thm:c_results_cographs}
We have $c_1^{\mathcal{CO}}(2)=8$ and $c_2^{\mathcal{CO}}(2)=11$.
\end{theorem}

\begin{proof}
Note that $c_1^{\mathcal{CO}}(2)\leq c_1^{\mathcal{CO}}(1)+5=8$ from Lemma \ref{lem:Straightupperbound}. Moreover, none of the extremal graphs in Theorem \ref{thm:c_1_2_PG} is a cograph, which implies $c_1^{\mathcal{CO}}(2)\geq 8$. On the other hand, since $c_2^{\mathcal{PG}}(2)=11$ from Theorem \ref{thm:c_2_2_PG}, we get $11\leq c_2^{\mathcal{CO}}(2)\leq c_2^{\mathcal{CO}}(1)+7=11$ by using Lemma \ref{lem:Straightupperbound}, so the result follows.
\end{proof}

\noindent Now, we have $c_0^{\mathcal{CO}}(2)=5$, $c_1^{\mathcal{CO}}(2)=8$ and $c_2^{\mathcal{CO}}(2)=11$. In what follows, we show that this trend holds for all $k\geq 3$.

%

\begin{theorem}
We have $c_k^{\mathcal{CO}}(2) = 3k+5$ for $k\geq 3$.
\end{theorem}

\begin{proof}
We first observe that we have $c_k^{\mathcal{CO}}(2)\leq c_k^{\mathcal{CO}}(1)+2(k+1)+1$ from Lemma \ref{lem:Straightupperbound}. Since $c_k^{\mathcal{CO}}(1)=k+2$, we obtain $c_k^{\mathcal{CO}}(2) \leq 3k+5$ for $k\geq 3$.

Let $G$ be a cograph of order $3k+5$. Since the complement of a cograph $G$ is also a cograph, and exactly one of $G$ or its complement $\overline{G}$ is disconnected \cite{cograph}, we assume without loss of generality that $G$ is disconnected. If $G$ has a $k$-defective $(2k+3)$-set, we are done since every $(k+2)$-set is $k$-defective. Therefore, assume $\alpha_k(G)\leq 2k+2$ and $\omega_k(G)\leq 2k+2$, and let $t$ be the number of connected components of $G$ whose size is at least $k+1$. If $t\geq 3$, then choosing $k+1$ vertices from each of these components forms a $k$-sparse $(3k+3)$-set, which is a contradiction. Besides, if $t=0$, then all vertices of $G$ is a $k$-sparse set. Thus, we need to examine two cases: $t=1$ and $t=2$.

Suppose $t=2$, and let $U$ and $V$ be the components of $G$ with $|U|\geq k+1$ and $|V|\geq k+1$. Now, if $G-(U\cup V)\neq \emptyset$, take a vertex $x\in G-(U\cup V)$ and choose $k+1$ vertices from each of $U$ and $V$. This forms a $k$-sparse $(2k+3)$-set, which is a contradiction. Hence, we can assume $G$ has only two components $U$ and $V$. Now, if one of $U$ and $V$ has a $k$-sparse $(k+2)$-set, without loss of generality say $U$, then a $k$-sparse $(k+2)$-set of $U$ together with $k+1$ vertices from $V$ form a $k$-sparse $(2k+3)$-set, a contradiction. Since  $R_k^{\mathcal{CO}}(i,k+2)=i$ for all $i\geq k+2$ (by \cite{defectiveramseynumbers}), both $U$ and $V$ are $k$-dense, and we are done.

Suppose $t=1$, and let $H$ be the connected component of $G$ with $|H|\geq k+1$. Since $G$ is disconnected, $G-H$ is non-empty, and it forms a $k$-sparse set. Moreover, if we can take a $k$-sparse set from $H$ and add it into $G-H$, we get a $k$-sparse set, too. Hence, we have $\alpha_k(H)+|G-H|\leq 2k+2$, which implies $\alpha_k(H)\leq |H|-k-3$. Similarly, consider the largest $k$-dense subset of $H$. If its size is at least $|H|-k-1$, we can add the remaining at most $k+1$ vertices of $H$ into $G-H$ to obtain a $k$-sparse set. This gives a $k$-defective 2-cocoloring of $G$, and we are done. So, assume $\omega_k(H)\leq |H|-k-2$. It follows that $\alpha_k(H)+\omega_k(H)\leq 2|H|-2k-5$, and since $|H|\leq 3k+4$, we have $\alpha_k(H)\leq 2k+1$ and $\omega_k(H)\leq 2k+2$. Since $k+2\leq \alpha_k(H)+1 \leq 2k+2$, if $\omega_k(H)\neq 2k+2$ then we have $k+2 \leq \omega_k(H)+1 \leq 2k+2$ and thus $R_k^{\mathcal{CO}}(\alpha_k(H)+1,\omega_k(H)+1)=\alpha_k(H)+\omega_k(H)-k$ from \cite{defectiveramseynumbers}. However, this gives $$1+|H|\leq R_k^{\mathcal{CO}}(\alpha_k(H)+1,\omega_k(H)+1)\leq 2|H|-3k-5,$$ implying that $|H| \geq 3k+6$, a contradiction with $|H|\leq 3k+4$. As a result, we have $\omega_k(H)=2k+2$, implying that $|H|-k-2\geq 2k+2$ and thus $|H|=3k+4$.

So, $H$ is a connected cograph on $3k+4$ vertices with $\omega_k(H)=2k+2$, and $G-H$ is an isolated vertex. Let us consider $\overline{G}$ which consists of the join of the disconnected cograph $F=\overline{H}$ having $\alpha_k(F)=2k+2$ with a single vertex, say $u$. We will show that $\overline{G}$ can be partitioned into two $k$-defective sets.

Let $s$ be the number of connected components of $F$ whose size is at least $k+1$. Since $\alpha_k(F)=2k+2$, we clearly get $s\in\{1,2\}$. If $s=2$, let $U_1$ and $V_1$ be components of $F$ with $|U_1|, |V_1|\geq k+1$. Since any subset of $k+1$ vertices from each of $U_1$ and $V_1$ form a $k$-sparse set and $\alpha_k(F)=2k+2$, we have $\alpha_k(U_1)=\alpha_k(V_1)=k+1$ and $F-(U_1\cup V_1)=\emptyset$. Again, by using $R_k^{\mathcal{CO}}(i,k+2)=i$ for all $i\geq k+2$ (by \cite{defectiveramseynumbers}), it follows that both $U_1$ and $V_1$ are $k$-dense. Since $u$ is adjacent to all vertices in $\overline{G}$, $U_1\cup \{u\}$ is $k$-dense too, thus $\overline{G}$ can be partitioned into two $k$-defective sets.

Suppose $s=1$ and let $A$ be the connected component of $F$ with $|A|\geq k+1$. Let $S$ be a largest $k$-sparse set in $A$, $K$ be a largest $k$-dense set in $A$, and write $B=S-A$. Notice that if $|A|-|S|\leq k+1$, then $(A-S)\cup \{u\}$ is $k$-dense and $S\cup B$ is $k$-sparse in $\overline{G}$. Similarly, if $|A|-|K|\leq k+1$, then $(A-K)\cup B$ is $k$-sparse and $K\cup \{u\}$ is $k$-dense. In both cases  $\overline{G}$ can be partitioned into two $k$-defective sets. Therefore, assume $|A|-\alpha_k(A),|A|-\omega_k(A)\geq k+2$, which gives $2|A|\geq \alpha_k(A)+\omega_k(A)+2k+4$. Moreover, since $B$ is nonempty, we get $|A|\leq 3k+3$ and so $k+1\leq \alpha_k(A),\omega_k(A)\leq 2k+1$. Hence, since $R_k^{\mathcal{CO}}(\alpha_k(A)+1,\omega_k(A)+1)=\alpha_k(A)+\omega_k(A)-k$ from \cite{defectiveramseynumbers}, we get $1+|A|\leq \alpha_k(A)+\omega_k(A)-k\leq 2|A|-3k-4$, which leads $|A|\geq 3k+5$ and so a contradiction. As a result, $\overline{G}$ can be partitioned into two $k$-defective sets, so can be $G$, we are done.
\end{proof}

\section{Conclusion}
In this work, we investigated the computation of defective Ramsey numbers and the parameter  $c_k^{\mathcal{G}}(m)$ in restricted graph classes. We obtained several results for perfect graphs, bipartite graphs, chordal graphs and cographs. Our approach combines efficient graph generation methods with classical direct proof techniques. We believe that the generic framework that we offer for efficiently generating structured graphs with desired properties is a promising approach in Ramsey theory, and more broadly in extremal graph theory. 

\section*{Acknowledgments}
The authors acknowledge the support of the The Scientific and Technological Research Council of Turkey (T{\"U}B\.{I}TAK) Grant no:118F397.

\section*{Declarations}

\paragraph{Funding}
This work has been supported by The Scientific and Technological Research Council of Turkey (T{\"U}B\.{I}TAK) Grant no:118F397.

\paragraph{Conflicts of interest/Competing interests}
The authors have no conflicts of interest to declare that are relevant to the content of this article.

\paragraph{Code availability}
The source codes are made available in public repositories, links to which are provided in the text and in the references.

\paragraph{Availability of data and material}
The extremal graphs output by our algorithms are made available in public repositories, links to which are provided in the text and in the references.

\end{document}